\documentclass[11pt]{amsart}
\usepackage{amssymb,latexsym,graphicx, amscd}
\usepackage{enumerate}
\newtheorem{theorem}{Theorem}[section]
\newtheorem{prop}[theorem]{Proposition}
\newtheorem{lemma}[theorem]{Lemma}
\newtheorem{remark}[theorem]{Remark}
\newtheorem{question}[theorem]{Question}
\newtheorem{definition}[theorem]{Definition}
\newtheorem{cor}[theorem]{Corollary}
\newtheorem{example}[theorem]{Example}

\begin{document}
%\today\par

\title{$J-$holomorphic  curves in a nef class}
\author{Tian-Jun Li }
\address{School  of Mathematics\\  University of Minnesota\\ Minneapolis, MN
55455}
\email{tjli@math.umn.edu}
\author{Weiyi Zhang}
\address{Mathematics Institute\\  University of Warwick\\ Coventry, CV4 7AL, England}
\email{weiyi.zhang@warwick.ac.uk}

\begin{abstract} Taubes established fundamental properties of  $J-$holomorphic subvarieties in dimension $4$ in \cite{T1}. 
In this paper, we further investigate  properties  of reducible $J-$holomorphic  subvarieties. We offer an upper  bound of the  total genus of  a subvariety when the class of the subvariety  is $J-$nef. 
For a spherical class, it has particularly strong consequences. 
It is shown that, for any tamed $J$, each irreducible component  is a smooth rational curve. It might be even new when $J$ is integrable.
We also  completely classify  configurations of maximal dimension. 
 To prove these results we treat subvarieties as weighted graphs and introduce  
several combinatorial moves.
\end{abstract}
 \maketitle

\tableofcontents
\section{Introduction}
Let $(M, J)$ be a closed, almost complex $4-$manifold. 
In this paper we study properties  of reducible $J-$holomorphic subvarieties in  $M$. 
Here $J$ is not always assumed to be  tamed.

\begin{definition} \label{subvariety} A closed set $C\subset M$ with finite, nonzero 2-dimensional
Hausdorff measure is said to be an irreducible  $J-$holomorphic subvariety if it has
no isolated points, and if the complement of a finite set of points
in $C$, called the singular points,  is a connected smooth submanifold with $J-$invariant tangent space.

A $J-$holomorphic subvariety $\Theta$ is a finite set of pairs $\{(C_i, m_i), 1\leq i\leq n\}$, where each  $C_i$ is irreducible 
$J-$holomorphic subvariety and each $m_i$ is a non-negative integer. 
The set of pairs is further constrained so that $C_i\ne C_j$ if $i\ne j$. 
\end{definition}

Pseudo-holomorphic subvarieties are closely related to, but  clearly different from pseudo-holomorphic maps.
They are the real analogues of one dimensional subvarieties in algebraic geometry. 
When $J$ is understood, we
will simply call a  $J-$holomorphic subvariety a subvariety. 
An irreducible  subvariety  is said to be {\it smooth} if it has no singular points. 
A subvariety $\Theta=\{(C_i, m_i)\}$ is said to be 
connected if $\cup C_i$ is connected. 

Taubes provides a  systematic analysis of pseudo-holomorphic subvarieties in \cite{T1}.  
The knowledge of the structure of reducible $J-$holomorphic subvarieties is very  important, in both the integrable case and the tamed case.  Among others, two aspects are especially significant for applications. Firstly,  under natural conditions, we need to know that the irreducible components are not too complicated. This point is used for example in the argument in \cite{GH} on the structure of rational curves. Secondly,  we need to know the moduli space of the reducible subvarieties is not too large to ensure the existence of irreducible subvarieties. This is used in \cite{LZ12} for the study of Donaldons's tamed-to-compatible question and almost K\"ahler Nakai-Moishezon criterion. These aspects are the main focus of this paper.

Suppose  $C$ is an irreducible subvariety.
Then it is the image of a $J-$holomorphic map
$\phi:\Sigma\to  M$ from a complex connected curve $\Sigma$, where $\phi$
is an embedding off a finite set. $\Sigma$ is called the model curve and $\phi$ is called the tautological map. 
The map $\phi$ is uniquely determined up to automorphisms of $\Sigma$. 
This understood, the associated homology class $e_C$ is defined to be the push forward of the fundamental class of $\Sigma$ via $\phi$. 
And for a subvariety $\Theta$, the associated class $e_{\Theta}$ is defined to be $\sum m_ie_{C_i}$.

A special feature in dimension 4 is that,  by the adjunction formula,   the genus of a smooth subvariety $C$ is given  by $g_J(e_C)$ defined as follows. 
Given a class $e$ in $H_2(M;\mathbb Z)$,
introduce  the $J-$genus of $e$, 
\begin{equation} \label{J-genus}
\begin{array}{lll}
 g_J(e)&=&\frac{1}{2}(e\cdot e+K_J\cdot e)+1,
 \end{array}
 \end{equation}
  where $K_J$ is the canonical class of $J$.

Moreover, when $C$ is irreducible,  $g_J(e_C)$ is non-negative.  In fact, if $\Sigma$
is the model curve of $C$, by the adjunction inequality in \cite{McD},  
\begin{equation}\label{adj inequality}
g_J(e_C)\geq g(\Sigma),
\end{equation}
%$g_J(e_C)$ is 
%bounded from below by the genus of the  model curve $\Sigma$ of $C$, 
with equality if and only if $C$ is smooth.

We investigate, under what conditions on the class $e$, 
$g_J(e)$ still bounds the total genus of 
any connected, reducible subvariety in $e$. 

\begin{definition}
The total genus $t(\Theta)$ of $\Theta$ is defined to be $\sum_i g_J(e_{C_i})$.
\end{definition}

\begin{question}
Suppose $e$ is a class with $g_J(e)\geq 0$ and $\Theta=\{(C_i,m_i)\}$ is a connected subvariety in the class $e$.
Find general conditions such that  
\begin{equation}\label{bound} g_J(e)\ge t(\Theta). 
 \end{equation}
 \end{question}

The study of Donaldson's ``tamed to compatible" question and almost K\"ahler Nakai-Moishezon duality by Taubes' subvariety-current-form strategy \cite{T1, LZ12} led us to this problem
in the case $g_J(e)=0$ and $J$ is tamed.
This  problem is very subtle when there are irreducible components 
with negative self-intersection and high multiplicity;  incorrect assertions are 
easily made from geometric intuition (see e.g. Example \ref{-2K}).

In this paper, we settle it for $J-$nef classes. A class $e$ is said to be {\it $J-$nef} if it pairs  non-negatively with any $J-$holomorphic subvariety. Whenever there is a $J-$holomorphic subvariety representative in a $J-$nef class, we have $e\cdot e\ge 0$.

\begin{theorem}\label{genus bound}
Suppose   $e$ is a $J-$nef class with $g_J(e)\geq 0$. Then  \eqref{bound} holds for any connected subvariety in the class $e$.  
\end{theorem}

To prove Theorem \ref{genus bound}, we treat subvarieties  as weighted graphs, and use  {\it curve expansion} and {\it curve combination}  to rearrange 
 the multiply covered part. In fact, these techniques are also effective analyzing  when the stronger bound 
 \begin{equation}\label{strong bound} g_J(e)\ge \sum_i m_ig_J(e_{C_i})
 \end{equation}
 holds.

Notice that when $g_J(e)=0$, 
we actually have equality in Theorem \ref{genus bound}.  This is because  $g_J(e_{C_i})\ge 0$ for all $i$ since each $C_i$ is  irreducible.  In turn this implies that $g_J(e_{C_i})=0$.
Moreover, we have the following more precise result.
\begin{theorem} \label{emb-comp}  Suppose  $e$ is a  $J-$nef class with $g_J(e)=0$. 
Let   $\Theta$ be a $J-$holomorphic subvariety in the class $e$. 
\begin{itemize}
\item 
If $\Theta$ is connected, then   each  irreducible component  of $\Theta$  is a smooth rational curve, and $\Theta$ is a tree configuration.
\item
If $J$ is tamed, then $\Theta$ is connected. 
\end{itemize}
\end{theorem}

Here, for a tree configuration, we refer to Definition \ref{deftree}. In particular, distinct components in a tree configuration intersect at most once.

Recall that $J$ is said to be tamed if there is
a symplectic form $\omega$ such that the bilinear form $\omega(\cdot, J(\cdot))$ is positive definite. 
The tameness is  necessary  for the second bullet since otherwise there could be a null homologous $J-$holomorphic torus in $\Theta$.

Thus, connected configurations in  a $J-$nef spherical class match our geometric intuition: each component is a smooth rational curve. 
A particularly nice consequence is 

\begin{cor} \label{embsphere} Suppose $J$ is a tamed almost complex structure  
and   $e$ is a class   represented by a smooth $J$-holomorphic rational curve. 
Then for  any $J-$holomorphic subvariety $\Theta$ in the class $e$,
%\begin{itemize}
%\item $\Theta$ is connected if $e$ is $J-$nef. 
%\item
each irreducible component of $\Theta$ is a smooth  rational curve. %\end{itemize}
\end{cor}

We will comment on various versions of this result  in the literature (\cite{GH}, \cite{McD}, \cite{McD1}) in \ref{rk-smooth}. 

For a $J-$nef spherical class,   the irreducible part of the moduli space, when non-empty,  is a smooth manifold of expected dimension. This is   due to the  ``automatic regularity" of any smooth rational curve  with non-negative self intersection.
In Corollary 4.10 we further show that the reducible part always has smaller dimension.  And if we assume that $J$ is tame,   by Proposition 4.5 in \cite{LZ12}, this assumption would actually guarantee the irreducible part of the moduli space to  be non-empty and hence the existence of a smooth rational curve in the given class. This is used as a crucial step in our study of Donaldson's tamed-to-compatible question and Nakai-Moishezon duality between the almost K\"ahler cone and the curve cone in \cite{LZ12}. Moreover, along with the techniques of \cite{LZ12}, the second author is able to apply our main results to study when a symplectic surface is symplectically isotopic to an algebraic curve in a general ambient $4-$manifold.

We also investigate which   stratum of the reducible part has codimension one. It is  interesting that, in this case, the curve combination moves we applied to prove Theorem \ref{genus bound} have a nice interpretation as combinatorial blow-downs. This viewpoint makes it possible 
 to classify the corresponding connected configurations in Theorem \ref{cd1conf} when $b^+=1$. Precisely, these configurations are shown to be either successive blow-ups of  a single smooth curve,
  or successive blow-ups of a comb configuration along the spike curve.
  
Finally, we would like to remark that, as pointed out by Gompf, the same arguments apply to closed holomorphic curves in a Stein manifold and all the results in this paper hold true as well. 

We appreciate useful discussions with D. McDuff,   W. Wu, and we  thank R. Gompf, V. Tosatti and M. Usher for their interest.  We are deeply grateful to the referee for careful reading and extremely helpful suggestions. 
The authors benefited from NSF grant 1065927 (of the first author). The second author is partially supported by AMS-Simons travel grant.

%%%%%%%%%%%%%%%%%%%%%%%%%%%%%%%%%%%%%%%%%%%%%%%%%%%%%%%%%%%

\section{Pseudo-holomorphic subvarieties}

We always assume $M$ is a $4-$dimensional manifold with a fixed almost complex structure $J$.

\subsection{Properties of irreducible subvarieties}

\subsubsection{Genus and adjunction number}
Let $C$ be an irreducible subvariety. 
The geometric genus of  $C$ is defined to be the genus of its model curve $C_0$, and 
the arithmetic genus of $C$ is $g_J(e_C)$. The adjunction inequality 
\eqref{adj inequality} says that the arithmetic genus is no less than the geometric genus.

The next  result follows directly from the adjunction inequality. 
\begin{lemma}\label{0}
 If $g_J(e_C)=0$, then $C$ is a smooth rational curve. 
\end{lemma}

It is  convenient to introduce  the adjunction number.

\begin{definition}
The adjunction number of $e$ is given by 
$$adj(e)=e\cdot e+K_J\cdot e.$$ 
\end{definition}

Notice that $$2g_J(e)=adj(e)+2.$$
By the adjunction inequality \eqref{adj inequality}, $adj(e_C)\geq -2$.

\subsection{The moduli space}
In this subsection we fix a class $e$. 

As in \cite{T1}, we define the moduli space $\mathcal M_e$ of subvarieties in the class $e$: Any element $\Theta$ in  $\mathcal M_e$  is a $J-$holomorphic 
subvariety with  $ e_{\Theta}=e$.

\begin{definition}\label{eff}
A  homology class $e\in H_2(M; \mathbb Z)$ is said to be  
$J-$effective if $\mathcal M_e$ is nonempty. 
\end{definition}

We use $\mathcal M_{irr, e}$  to denote the moduli space of irreducible subvarieties in class $e$. Let $\mathcal M_{red,e}$ denote $\mathcal M_e \setminus \mathcal M_{irr,e}$.

\subsubsection{Topology}
$\mathcal M_e$ has a natural topology. Let $|\Theta|=\cup_{(C, m)\in \Theta}C$ denote the support of $\Theta$. Consider  the symmetric, non-negative function, $\varrho$, on $\mathcal M_e\times \mathcal M_e$ that is defined by the following rule:
\begin{equation} \varrho(\Theta, \Theta')=\sup _{z\in |\Theta|} \hbox{dist}(z, |\Theta'|)
+\sup _{z'\in |\Theta'|} \hbox{dist}(z', |\Theta|).
\end{equation}
The function $\varrho$ is used to measure distances on $\mathcal M_e$, where the distance function dist is defined by an almost Hermitian metric on $(M, J)$.

Given a smooth $2$-form $\nu$ we introduce the pairing
$$(\nu, {\Theta} )=\sum_{(C, m)\in \Theta} m\int_{C}\nu.$$

The topology on $\mathcal M_e$ is defined in terms of convergent sequences:

A sequence $\{\Theta_k\}$ in $\mathcal M_e$ converges to a given element $\Theta$ if the following two conditions are met:

\begin{itemize}
\item  $\lim_{k\to \infty} \varrho (\Theta, \Theta_k)=0$.

\item  $\lim_{k\to \infty} (\nu, \Theta_k)=(\nu, \Theta)$ for any given smooth 2-form $\nu$.
\end{itemize}

\begin{definition} Given a class $e$, introduce  its $J-$dimension, 
\begin{equation} \label{l}  \iota_e=\frac{1}{2}(e\cdot e-K_{J}\cdot e). 
\end{equation}
\end{definition}

$\iota_e$ is the expected dimension of the moduli space $\mathcal M_e$.

\subsubsection{Smooth rational curves}

When $e$ is a class  represented by a smooth  rational curve,
we introduce 
$$l_e=\max\{\iota_e, 0\}.$$

The following is an immediate consequence of the adjunction formula and the adjunction inequality \eqref{adj inequality}. 

\begin{lemma} \label {irr=smooth} If $g_J(e)=0$, then  

$\bullet$ $\iota_e=e\cdot e+1$, where $\iota_e$ is defined in  \eqref{l};

$\bullet$ every element in  $\mathcal M_{irr,e}$  is a smooth rational curve. 
\end{lemma}

One special feature of the moduli space of smooth rational curves is the following automatic transversality 
(\cite{HLS97}), 
which is valid for an arbitrary almost complex structure. 

\begin{lemma} \label{pdmanifold} Let $e$ be a class  represented by a smooth  rational curve with $e\cdot e\geq -1$.
Then $\mathcal M_{irr,e}$ is a smooth manifold of dimension $2l_e$. 

\end{lemma}

\subsection{$J-$nef class}

\begin{definition} \label{nef}
A  homology class $e\in H_2(M; \mathbb Z)$ is said to be $J-$nef if it pairs non-negatively with   any $J-$effective class.

\end{definition}

The following lemma immediately follows from the positivity of intersections of distinct irreducible subvarieties.

\begin{lemma}\label{irr=nef}
If $e$ is   represented by an irreducible $J-$holomorphic subvariety and $e\cdot e\ge 0$, then $e$ is  a $J-$nef class.

On the other hand, if $e$ is $J-$nef and $J-$effective, $e\cdot e\ge 0$.
\end{lemma}

For a tamed almost complex structure $J$, the notion $J-$nef is a 
natural condition which guarantees the good behavior of the $J-$holomorphic subvarieties, as can be seen  in many results in this paper and  also Examples \ref{-2K} and \ref{not connected}. Moreover, it is  an important condition for constructing non-negative currents in \cite{LZ12} as we briefly explain below.
%is introduced in \cite{LZ12} which aims at a similar role of nef classes in K\"ahler geometry.

 It is known that,
when $J$ is K\"ahler, in any big and nef cohomology class ({\it i.e.} a nef cohomology class with positive top power), there is a K\"ahler current. 
Here a current is a differential form with distribution coefficients. Hence it represents a real cohomology class when pairing with smooth closed forms in the weak sense.
The K\"ahler currents play an intermediate role in \cite{DP} to construct K\"ahler forms. For the subvariety-current-form strategy, Taubes current is such an intermediate object, which is usually constructed through integrations over certain moduli space of subvarieties. Hence, our definition of $J-$nef mimics the original algebraic one, instead of the K\"ahler notion. Notice our definition does not require the existence of an almost K\"ahler form.

\subsection{When $J$ is tamed}

Here is a well known fact that we will need in Section 4. 

\begin{lemma}\label{non-trivial}
If  $J$ is tamed then  the  homology class $e_C$ of any subvariety $C$ is nontrivial.
\end{lemma}

Here is a simple consequence of this fact and positivity of intersection. 

\begin{lemma}\label{uniqueness}
Suppose $e\cdot e<0$ and  $\mathcal M_{irr, e}$ is nonempty. Then  $\mathcal M_{e}$
consists of a unique element. 
\end{lemma}

\begin{proof}
Let  $C$ be an irreducible variety in the class $e$. Suppose $\Theta=\{(C_i, m_i)\}$ is any subvariety in 
the class $e$. If each $C_i$ is distinct from $C$, then $e\cdot e=e_C\cdot \sum m_ie_{C_i}$
is non-negative. This is impossible. 

Suppose $C=C_i$ for some $i$, say $i=1$.  Then by Lemma \ref{non-trivial}, the subvariety 
$\Theta'=(C_1, m_1-1)\cup \{(C_i, m_i), i\geq 2\}$ is empty, namely, $m_1=1$ and $\Theta=\{(C_1, 1)\}$. 
\end{proof}

Another basic fact is  that  $\mathcal M_e$ is compact by the Gromov compactness.  

\subsubsection{$K_J-$spherical classes are $J-$effective}

Let $S$ be the set of homology classes of $M$ which are represented by smoothly embedded spheres.
 
The set of  $K_J-$spherical classes is defined to be  $$S_{K_J}=\{e\in S|g_J(e)=0\}.$$
\begin{prop}\label{existence}   
Let  $e$ be a class in $S_{K_J} $.
\begin{itemize}
\item  Suppose $e\cdot e\geq -1$. Then  for any symplectic form $\omega$ taming $J$, the  Gromov-Taubes  invariant 
of $e$ is nonzero. In particular, $\mathcal M_e$ is nonempty, i.e. $e$ is $J-$effective. 

\item If  $e\cdot e\ge 0$, then $M$  has to be rational or ruled, which has $b^+=1$.
\end{itemize}
\end{prop}

\begin{proof}
The first statement  is a consequence of Taubes' symplectic Seiberg-Witten theory, see e.g.  \cite{LL1}. 

The second statement  follows the first statement and   \cite{McD}. 
\end{proof}

%\subsubsection{Uniqueness of  subvariety with negative self-intersection}

%%%%%%%%%%%%%%%%%%%%%%%%%%%%%%%%%%%%%%%%%%%%%
\section{Bounding the total genus}

In this section, we prove Theorem \ref{genus bound}. 
\subsection{Two simple cases}

Suppose $\Theta=\{(C_1, m_1),\cdots, (C_n, m_n)\}$. Let $e_i=e_{C_i}$.

\subsubsection{Multiplicity one}
We first deal with the  case where each $m_i$ is equal to one.

\begin{lemma}\label{simple}
Suppose $\Theta=\{(C_1, 1), \cdots, (C_n, 1)\}$, then both \eqref{bound} and \eqref{strong bound} hold. 
\end{lemma}

\begin{proof}
We compare   the adjunction numbers:
$$
adj(e)=\sum_i adj(e_i) +  \sum_{i\ne j}e_i\cdot e_j.
$$
By the adjunction inequality \eqref{adj inequality}, $adj(e_i)\geq -2$. By the positivity of intersections, $e_i\cdot e_j\geq 0$ for any $i\ne j$.

If there are $l$ components, then there are at least $l-1$ transversal intersection points.
Thus 
\begin{equation}\label{simple estimate} 
\begin{array}{ll}
2g_J(e)=adj(e)+2&=\sum_i adj(e_i) +  \sum_{i\ne j}e_i\cdot e_j+2\cr

&\geq \sum_{i=1}^l adj(e_i)+2(l-1)+2 \cr

&=\sum_{i=1}^l(adj(e_i)+2)=2\sum g_J(e_i).
\end{array}
\end{equation}

\end{proof}

\subsubsection{One component}
Next we deal with the case that there is only one component.

\begin{lemma}\label{1vertex}
Suppose  $\Theta=\{(C_1, n)\}$ with  $n>1$. 
Both  \eqref{strong bound} and \eqref{bound} hold if $e_1\cdot e_1>0$. 
When $e_1\cdot e_1=0$, \eqref{bound} holds if $g_J(e_1)\geq 1$.

\end{lemma}
\begin{proof} 
$$2[g_J(e)- n\, g_J(e_1)] =
 (n^2-n)e_1\cdot e_1+(2-2n)$$
 If $e_1\cdot e_1>0$, then
$g_J(e)- n\, g_J(e_1) \ge (n-1)(n-2)\ge 0.$

When $e_1\cdot e_1=0$,
$$2[g_J(e)- g_J(e_1) ]=
 (n-1)K_J\cdot e_1=2(n-1)(g_J(e_1)-1).$$
\end{proof}

On the other hand, if $e_1\cdot e_1<0$, then  \eqref{strong bound} always fails and 
\eqref{bound} could fail.

\begin{example}\label{-2K}
Suppose $M=\mathbb C \mathbb P^2\#10\overline{\mathbb C \mathbb P^2}$ and there is a smooth $J-$holomorphic genus one curve $C$ in $-K_J$. Then the
subvariety $\Theta=\{(C, 2)\}$ fails \eqref{bound} since $g_J(-2K_J)=0$ and $t(\Theta)=1$. 
\end{example}

The multiplicity one  case  and the one component case are settled,  even without the $J-$nef assumption.

We next introduce moves to reduce the general case to these two simple cases.
To better describe these moves and their properties  we view  reducible curves as  graph like objects, and introduce  curve configurations.

\subsection{Nef, connected weighted graphs}
\begin{definition} Here  a weighted graph refers to a graph whose vertices are weighted 
by a pair of a $J-$effective  class $\in H_2(M;\mathbb Z)$ and a positive integer  multiplicity. 

The  edges are determined by the weighted vertices: there is an edge connecting two vertices whenever 
the intersection number of their classes is nonzero. 
Further, label each edge by the intersection number   of the classes of its vertices.

The adjunction number and the self-intersection number of each vertex are those of its homology class. 
\end{definition}

\begin{definition} A curve configuration  is  a weighted graph
satisfying the following two properties:
\begin{itemize}
 \item the adjunction number of each vertex  is at least $-2$.
 
 \item  the  label of each edge  is positive. 
\end{itemize}
\end{definition}

Specifically, to each reducible curve, we 
assign a weighted graph as follows:
to  each component $C_i$, assign the vertex, still denoted by $C_i$,  weighted by the pair
$(e_i, m_i)$.

Notice that for  each pair of intersecting components $C_i, C_j$, there is an edge connecting the corresponding vertices  labeled
by their intersecting number, and all edges arise this way.
Clearly, the resulting weighted graph is a curve configuration due to the adjunction inequality \eqref{adj inequality} and the  positivity of intersection. Moreover, 
the  curve configuration is connected as a graph if and only if the reducible curve is connected. 

Introduce the total class of a weighted graph in the obvious way. 
The adjunction number (resp. $J-$genus) of
a weighted graph is then defined to be the adjunction number (resp. $J-$genus) of its
total class. 

\begin{definition}\label{nef graph}
A weighted graph is said to be nef if its total class pairs non-negatively with the class of
each vertex.

\end{definition}

Here is an example of a nef curve configuration with total class $e$ which is not the graph 
corresponding to a $J-$holomorphic reducible curve in class $e$. 

\begin{example} \label{referee}
Suppose $M=S^2\times S^2$ with spherical classes $a=[S^2\times pt]$ and $b=[pt \times S^2]$. Let $J$ be such that $a-2b$ has a $J-$holomorphic representative. Then $e=a+2b$ is a $J-$nef class. The graph with two vertices, one labeled by $(a-b, 1)$ and the other by $(b, 3)$ is a nef curve configuration, but there is no corresponding $J-$holomorphic reducible curve. 
\end{example}

Lemmas \ref{simple} and \ref{1vertex} now takes the following form,
\begin{lemma}\label{simple-graph}
Given  a connected curve configuration,  
if the  multiplicity of each vertex is $1$,  then the sum of $J-$genus of vertices is  bounded from above by the $J-$genus of its total class.

Given any nef curve configuration  with only one vertex weighted by $(e_1, n)$, let  $e=ne_1$. Then $g_J(e) \geq  g_J(e_1)$  if $g_J(e) \geq 0$, and $g_J(e) \geq  n  g_J(e_1)$ if $e_1\cdot e_1> 0$.

\end{lemma}

\begin{proof}  The first statement is exactly a  rephrase of Lemma \ref{simple} in the weighted graph language. 

For the second statement, it follows from Lemma \ref{1vertex} and the following observation: 
By Lemma \ref{irr=nef}, $e_1\cdot e_1\ge 0$ 
since $e_1$ is a $J-$effective class and the weighted graph is  nef.    
\end{proof}
 
And Theorem \ref{genus bound} follows from
 
  \begin{prop}\label{emb-comp-graph}
Given a connected, nef curve configuration whose total class has non-negative $J-$genus,  then the sum of $J-$genus of vertices is  bounded
from above by the $J-$genus of its total class. 
 
\end{prop}
 
 %%%%%%%%%%%%%%%%%%%%%%%%%%%%%%%%%%%%%%%%%%%%%%%%%%%%%%%%%%%%%%%%%%
 \subsection{Curve expansion and curve combination}
\subsubsection{Curve expansion} We start with moves on  vertices with non-negative self-intersection. 

Given a weighted graph, for each  vertex  $C$ with weight $(e_C, m)$ such that $e_C\cdot e_C\geq 0$
and $m>1$,   replace it by  $m$ vertices, $C(k), 1\leq k\leq m$, 
weighted by $(e_C, 1)$. This operation is called curve expansion.

\begin{lemma}\label{simple'} Given a connected curve configuration with at least two vertices, the expanded weighted graph is still a connected curve configuration. If the original configuration is nef, so is the new one. 

The sum $\sum_i g_J(e_i)$ is always non-decreasing. The sum $\sum_i m_ig_J(e_i)$ 
is non-decreasing if curve expansion is not applied to vertex  $C$ with weight $(e_C, m)$ such that $e_C\cdot e_C=0$, $m>1$ and $g_J(e_C)>0$.

Consequently, 
Proposition \ref{emb-comp-graph} is true if the multiplicity of each  vertex
with negative self-intersection is $1$.
\end{lemma}
\begin{proof} Consider  the expanded curve configuration.

Notice that the new vertices $C(k)$ have the same first weight and then the same adjunction number as that of $C$. 

There are two kinds of new edges. 
If there is an edge connecting $C$ with another vertex $D$ in the original curve configuration, then there is an edge joining $D$ with each $C(k)$ by an edge with the same positive label. 
Therefore the resulting weighted graph is connected. 
If  the self-intersection number of $C$ is positive, then there is an edge joining each pair of $C(k)$. 
Since the labels of these edges are also positive,   
 the resulting weighted graph is  a connected curve configuration with the same total class and the same total multiplicity. 

The genus estimates essentially follow from Lemma \ref{1vertex}. The inequality $g_J(me_C)\ge  g_J(e_C)$ always holds when $e_C\cdot e_C\ge 0$. Hence $\sum_i g_J(e_i)$ is non-decreasing. If we are not applying expansion for $(e_C, m)$ with $e_C\cdot e_C=0$, $m>1$ and $g_J(e_C)>0$, the strong inequality $g_J(me_C)\ge  mg_J(e_C)$ holds, which implies $\sum_i m_ig_J(e_i)$ 
is non-decreasing.
\end{proof}

Thus  we may assume all the vertices with non-negative self-intersection have multiplicity $1$.   

Next we deal with  vertices with negative self-intersections, especially $-1$ vertices.  
Here a vertex is called a $-1$ vertex if its class has self-intersection $-1$.

\subsubsection{Curve combination} \label{notation}
Given a connected curve configuration with the property that any vertex with multiplicity greater than $1$ has negative self-intersection.  

\begin{enumerate}[(i)]
\item Suppose there are two adjoined vertices $V_1, V_2$ weighted by $(D_i, n_i)$ with   $n_1=n_2=n$. 
Collapse  them to a  vertex $V$  weighted by
$(D_1+D_2, n)$. We call this move (i)$_n$.

\item Suppose there are two adjoined  vertices $V_1, V_2$ weighted by $(D_i, n_i)$  with $n_1>n_2$, and 
  $D_1\cdot D_2\ge -D_1\cdot D_1$. 
  Replace them 
  by two vertices $V, V'$ weighted by $(D_1+D_2, n_2)$ and $(D_1, n_1-n_2)$ respectively.

\item Suppose there  is a $-1$ vertex  $E$ with multiplicity $n_0$, and  there are neighboring vertices  weighted by $(D_i, n_i), 1\leq i\leq t,$ with  $D_i\cdot D_i\leq -2, 1\leq i\leq t$
 and $$n_1D_1\cdot E+\cdots+n_tD_t\cdot E= n_0.$$ 
 Replace  them by  $t$ vertices weighted by  $(D_i+(D_i\cdot E) E,   n_i), 1\leq i\leq t$. 
 Notice that here we allow  $n_i= 1$.
 
 To record the  value of $t$, we sometimes call this move (iii)$_t$.
\end{enumerate}

The following simple observation is crucial for us:
\begin{lemma}\label{adj=-2}
If we apply any of the three moves above to  a connected, nef curve configuration, the new  weighted graph is a connected, nef  curve configuration
with the same total class. Moreover, it has  the following properties:

\begin{itemize}

\item  The sum of the multiplicities  of vertices   gets smaller. 
 
\item   The sum  $\sum_i g_J(e_i)$ is non-decreasing for any curve combination 
move.

\item  $\sum_i m_ig_J(e_i)$ is also non-decreasing for any curve combination 
move.
\end{itemize}
 \end{lemma} 
\begin{proof} Firstly, we notice that the first weight of each vertex is still a $J-$effective class since it is a linear combination of that of old vertices with non-negative coefficients.

To show that  the new configuration is a curve configuration,
we first verify  the adj condition: 
\begin{equation}\label{move12}
adj(D_1+D_2)\ge 2+ adj(D_1) + adj(D_2) \ge -2
\end{equation}
 for moves (i) and (ii), and 
\begin{equation}\label{move3adj}
adj(D_i+(D_i\cdot E)E)=adj(D_i)+(D_i\cdot E)^2+(2g_J(E)-1)D_i\cdot E\ge -2
\end{equation}
 for move (iii). 

Next we verify the label condition. Clear for moves (i) and (ii). For move (iii),  the label of each new edge 
is 
 $$(D_i+(D_i\cdot E) E)\cdot (D_j+(D_j\cdot E) E)=D_i\cdot D_j+(D_i\cdot E)\cdot (D_j\cdot E)>0.$$

Let us prove the curve configuration is connected. It is clear for move (i). For move (iii), consider
the collection of new vertices.  The sum of their classes is 
the sum of the classes of the replaced vertices, so at least one of new vertices 
  is  connected to the rest of the configuration. Moreover, any  two new 
vertices are adjoined to each other since we have shown that $(D_i+(D_i\cdot E) E)\cdot (D_j+(D_j\cdot E) E)>0.$

For move (ii), we need the nefness condition. If $V_1$ is connected to another vertex
 in the original configuration other than $V_2$, then the new vertex 
 $V'$ is adjoined to  the same vertex. Hence, the new configuration is connected. 
Otherwise, only $V_2$ is connected to other vertices. 
The graph is assumed to be nef, thus 
$$e\cdot D_1=(n_1-n_2)D_1\cdot D_1+n_2(D_1\cdot D_1+D_1\cdot D_2)\ge 0,$$ 
which implies $$D_1\cdot (D_1+D_2)\ge \frac{n_1-n_2}{n_2}(-D_1\cdot D_1)>0.$$
 This shows the new configuration is still connected since 
$V$ is connected to the vertices that $D_2$ was connected to.

For the first bullet of the properties, the sum of multiplicities are reduced by 
$n_2$ for the first two moves, and $n_0$ for the third. 

For the second and the third bullets, the conclusion for 
first two moves follows from \eqref{move12}. For the third move, 
\eqref{move3adj} implies 
$$adj(D_i+(D_i\cdot E)E)\ge adj(D_i)+2(D_i\cdot E)g_J(E).$$
This shows 

$$\begin{array}
  {ll}
&\sum_{i=1}^tn_i g_J(D_i+(D_i\cdot E)E)\\
\\
&\ge \sum_{i=1}^t(n_ig_J(D_i)+n_iD_i\cdot Eg_J(E)) \\
\\
&=\sum_{i=1}^tn_ig_J(D_i)+n_0g_J(E).
\end{array}
$$
Similarly, $\sum_{i=1}^tg_J(D_i+(D_i\cdot E)E)\ge \sum_{i=1}^tg_J(D_i)+g_J(E)$.
\end{proof}

Here is an example how to apply these   moves. 

\begin{example}
Consider the curve configuration in $\mathbb C \mathbb P^2\# 5\overline{\mathbb C \mathbb P^2}$ with $4$ vertices weighted by 
 $$(H-E_1-E_2-E_3, n), (H-E_1-E_4-E_5, n), (E_1, 2n), (2H-E_2-E_3-E_4, 1).$$ 
 The total class is  $$(2n+2)H-(n+1)E_2-(n+1)E_3-(n+1)E_4-nE_5,$$ which has $J-$genus $0$ and is Cremona equivalent to $(n+1)H-nE_1$. Here Cremona equivalence refers to the equivalence 
 under the group of  diffeomorphisms preserving the canonical class $K_J$. 
 
 First apply move  (iii) to the $-1$ vertex $(E_1, 2n)$
  to obtain the curve configuration   with $3$ vertices weighted by 
  $$(H-E_2-E_3, n), (H-E_4-E_5, n), (2H-E_2-E_3-E_4, 1).$$ Then apply move (i) to the first two vertices to obtain 
the curve configuration   with $2$ vertices weighted by   
$$(2H-E_2-E_3-E_4-E_5, n), (2H-E_2-E_3-E_4, 1).$$

\end{example}

\subsection{Nef, connected curve configuration with at least two vertices}
\subsubsection{Rearrangement}
\begin{lemma}\label{combination} Suppose a connected, nef curve configuration has at least two vertices. 
After applying curve expansion and  appropriate curve combination moves (i), (ii), (iii) to  $-1$ vertices,  we would end up with a connected, nef curve configuration such that 
\begin{itemize}

\item All vertices 
with non-negative self-intersection have multiplicity $1$;

\item The $-1$ vertices  are not adjoined  to each other. Moreover,  any $-1$ vertex is not adjoined to a vertex  with non-negative 
self-intersection;
\item If vertices weighted by $(D_i, n_i), 1\leq i\leq t$  with  $D_i\cdot D_i\leq -2$ are all adjoined to a $-1$ vertex  $(E, n_0)$, then 
$$n_1D_1\cdot E+\cdots+n_tD_t\cdot E>n_0.$$
\end{itemize}
\end{lemma}

\begin{proof}
We apply move (i)  first to each $-1$ vertex. 
After this is done we could assume that, for any $-1$ vertex, its multiplicity $m$ is different from the multiplicity of any adjoined vertex. 

We now apply move (ii) to each $-1$ vertex  whenever it   is adjoined to a vertex  with   self-intersection at least $-1$.

After  applying moves (i) and (ii) repeatedly,  we could assume  that the second bullet is valid. 

Given a $-1$ vertex weighted by  $(E, n_0)$, suppose the  vertices that are adjoined to it are weighted by  $(D_i, n_i), 1\leq i\leq t$. Observe that, by the second bullet, each $D_i$ has
self-intersection $\le -2$ and 
$$e\cdot E=n_1D_1\cdot E+\cdots+n_tD_t\cdot E-n_0.$$ 
Since the total class $e$ is $J-$nef, we have a priori that $$n_0\leq n_1D_1\cdot E+\cdots+n_tD_t\cdot E.$$

If $n_0=n_1D_1\cdot E+\cdots+n_tD_t\cdot E$, we are then in the situation to apply move  (iii).
 This move may actually produce new $-1$ vertices and even vertices with non-negative self-intersection. 
If so, we apply curve expansion and curve combination moves (ii) and (i) again to rearrange so that the 
 first and the second bullets are valid.

We notice that  such rearrangement would stop in finite steps. This is because:  (a)    the total multiplicity is 
  preserved after curve expansion, and it is reduced after each curve combination by the  first bullet of  Lemma \ref{adj=-2}, so we could only apply finitely many curve combination moves, 
  (b) between two curve combination moves, the number of  curve expansions is bounded by  the total multiplicity.

\end{proof}

\subsubsection{After rearrangement}

\begin{lemma}\label{multi 1} For a connected, nef curve configuration satisfying all the three bullets in Lemma \ref{combination}, 
if there is a vertex with  multiplicity greater than $1$, then  
 the $J-$genus satisfies
$$g_J(e)\ge 1+ \sum_i m_ig_J(e_i).$$

\end{lemma}

This technical lemma will be proved in the next subsection.

\begin{example}

 Here is one  example of a connected, nef curve configuration satisfying all the three bullets in Lemma \ref{combination}, 
and having a vertex with  multiplicity greater than $1$: 
   $\Theta=\{(C_1,2), (C_2, 1), (C_3, 1), (C_4, 1), (C_5, 1)\}$ with 
 $$ e_{C_1}=H-E_1-E_2-E_3, \quad e_{C_2}=e_{C_3}=e_{C_4}=H, \quad e_{C_5}=E_1.$$
Here $e=5H-E_1-2E_2-2E_3$, and $g_J(e)=4$.  

\end{example}

%%%%%%%%%%%%%%%%%%%%%%%%%%%%%%%%%%%%%%%%%%%%%%%%%%%%%%%%%%%%
\subsection{Proof of Proposition \ref{emb-comp-graph} and Lemma \ref{multi 1}} 
We first prove Proposition \ref{emb-comp-graph}, which assumes Lemma \ref{multi 1}.
\subsubsection{Proposition \ref{emb-comp-graph}}
 \begin {proof}  [Proof of Proposition \ref{emb-comp-graph}] 
 One vertex curve configuration case follows from the second half of Lemma \ref{simple-graph}. Hence we assume there are at least two vertices.

 Denote the curve configuration by $G=\{(e_i, m_i)\}$. We apply the moves to get a curve configuration $G'=\{(e_j', m_j')\}$ as in Lemma \ref{combination}. 
By the second bullet  of Lemma \ref{adj=-2} and Lemma \ref{simple'}, the sum $\sum_i g_J(e_i)$ is non-decreasing for any curve expansion and curve combination 
move.

   By Lemma \ref{multi 1}, if there is a vertex with multiplicity greater than $1$, 
then 
$$\begin{array} {ll}g_J(e)&\geq \sum_j m_j'g_J(e_j')+1\\
\\
&\geq \sum_j g_J(e_j')+1\\
\\
&\ge \sum_i g_J(e_i)+1.
\end{array}$$
If the multiplicity of each vertex is $1$, apply  the first statement of Lemma \ref{simple-graph} instead of Lemma \ref{multi 1}, we obtain similarly $g_J(e)\geq \sum_i g_J(e_i)$.

\end{proof}

\subsubsection{A stronger bound}
In fact, we can establish the stronger estimate
  $$ g_J(e)\ge \sum_i m_ig_J(e_{C_i}),$$
if there is no vertex having class $me'$ with $e'\cdot e'=0$, $m\geq 2$ and $g_J(e')\geq 1$ in any intermediate step of the rearrangement.

First of all,  during the arrangement, if we never need to apply curve expansion for vertex  $C$ with weight $(e_C, m)$ such that $e_C\cdot e_C=0$, $m>1$ and $g_J(e_C)>0$, 
the sum $\sum_i m_ig_J(e_i)$ is non-decreasing by the third bullet of Lemma \ref{adj=-2} and Lemma \ref{simple'}. 

After the arrangement, if there is a vertex with multiplicity greater than $1$, 
then by Lemma \ref{multi 1}, 
$$\begin{array} {ll}g_J(e)&\geq \sum_j m_j'g_J(e_j')+1\\
\\
&\geq \sum_i m_ig_J(e_i)+1.\\
\end{array}$$
If the multiplicity of each vertex is $1$, apply the first statement of  Lemma \ref{simple-graph} instead of Lemma \ref{multi 1}, we obtain $g_J(e)\geq \sum_i m_ig_J(e_i)$.

\subsubsection{Lemma \ref{multi 1}}
It remains to prove Lemma \ref{multi 1}.

\begin{proof} [Proof of Lemma \ref{multi 1}]
For a configuration as in Lemma \ref{combination}, we suppose there are $l$ vertices weighted by  $(e_i, m_i)$. Moreover, we define $s_1, s_2, s_3$ by requiring that
\begin{itemize}
\item If $1\le i\le s_1$, then $m_i\ge 2$ and $e_i\cdot e_i\le -2$;
\item If $s_1+1\le i\le s_1+ s_2$, then $m_i\ge 2$ and $e_i\cdot e_i=-1$;
\item if $s_1+s_2+1\le i\le s_1+s_2+  s_3$, then $m_i=1$ and $e_i\cdot e_i\le -1$;
\item If $s_1+s_2+ s_3+1\le i\le l$, then $m_i=1$ and $e_i\cdot e_i\ge 0$.
\end{itemize}

We further let 
$$s=s_1+s_2.$$
With this understood we set up to show that if $s>0$ then 
$$adj(e)\ge \sum_i m_i(adj(e_i)+2).$$

If $1\leq i\leq s$, write $m_ie_i=\sum_{1\le k\le m_i}e_i(k)$, with each $e_i(k)=e_i$.
$$\begin{array}{ll}
adj(e) &=\sum_{1\leq i\leq s} \sum_{1\leq k\leq m_i} adj(e_i(k))\\
\\
&+ \sum_{1\leq i\leq s} \sum_{1\leq k\ne k'\leq m_i} e_i(k)\cdot e_i(k')+  \sum_{1\leq i\leq s} (m_ie_i) (e-m_ie_i)\\
\\
&+\sum_{j>s} adj(e_j)+ \sum_{j>s} e_j\cdot \sum_{i\leq s} m_ie_i+ \sum_{ j, k>s, j\ne k} e_j\cdot e_k.\\ 

\end{array}
$$

The adjunction terms satisfy
$$
\sum_{1\leq i\leq s} \sum_{1\leq k\leq m_i} adj(e_i(k))+ \sum_{j>s} adj(e_j)=  \sum_{i=1}^l m_i\cdot adj(e_i).
$$

We claim that  the cross terms satisfy $$\sum_{j>s} e_j\cdot \sum_{i\leq s} m_ie_i+ \sum_{ j, k >s,  j\ne k} e_j\cdot e_k\geq  2(l-s).$$
To justify the claim, introduce $\alpha=e_1+\cdots+e_s$, and  rewrite  
as  \begin{equation}\label{cross}
       \sum_{j, k>s,  j\ne k  } e_j\cdot e_k+2\sum_{j>s} e_j\cdot  \alpha    +\sum_{j>s} e_j\cdot \sum_{i\leq s} (m_i-2)e_i.
\end{equation}
Since $m_i\geq 2$ for $i\leq s$, the last term of (\ref{cross}) is non-negative. 

To estimate the first two terms of (\ref{cross}),  view the portion of the configuration  involving vertices with 
 $i\le s=s_1+s_2$ as one single vertex $(\alpha, 1)$. Notice here we use the assumption $s>0$. Along with the remaining $l-s$ vertices, 
we obtain a graph with $l-s+1$ vertices. 
Notice that this graph is still connected. 

Twice of the total labeling of this new graph  is exactly the sum of the first two terms. For this graph of $l-s+1$ vertices  to be connected, we need at least $l-s$ edges.
Since each label is positive, we obtain the desired estimate.

The remaining terms
\begin{equation} \begin{array}
  {ll}
&\sum_{1\leq i\leq s} \sum_{1\leq k\ne k'\leq m_i} e_i(k)\cdot e_i(k')+  \sum_{1\leq i\leq s} (m_ie_i) (e-m_ie_i)\\
\\
&=\sum_{1\leq i\leq s} m_i(m_i-1)e_i\cdot e_i +\sum_{1\leq i\leq s} (m_ie_i)\cdot (e-m_ie_i) \\
\\
&=\sum_{1\le i\le s} m_ie_i(e-e_i).

\end{array}
\end{equation} 

Sum them up, and notice that $m_i=1$ if $i\ge s$, we have 
$$\begin{array}{ll}
adj(e) &\ge \sum_{i=1}^l m_i\cdot adj(e_i)\\
\\
&+ \sum_{1\le i\le s} m_ie_i(e-e_i)\\
\\
&+2(l-s)\\
\\
&=\sum_{i=1}^l m_i\cdot (adj(e_i)+2) \\
\\
&+\sum_{1\le i\le s} m_ie_i(e-e_i)-\sum_{1\le i\le s}2m_i.
\end{array}
$$
 So it suffices to show that 
$$\sum_{1\le i\le s} m_ie_i(e-e_i)-\sum_{1\le i\le s}2m_i\ge 0.$$

We separate the discussion into the two cases $i\le s_1$ and $s_1+1\le i\le s$.

{\bf Case I:} When $i\le s_1$, since the curve configuration  is nef,

$$\sum_{1\le i\le s_1} m_ie_i(e-e_i)-\sum_{1\le i\le s_1}2m_i\ge \sum_{1\le i\le s_1} -m_ie_i\cdot e_i-\sum_{1\le i\le s_1}2m_i\ge0.$$
 The last inequality holds because $m_i\ge 2$ and $C_i^2\le -2$  if $1\le i\le s_1$.

{\bf Case II:}

For $s_1+1\le i\le s$, we need to be more careful.

$$e\cdot e_i=(\sum_{j\le s_1} m_je_j+\sum_{s+1\le j \le s+s_3}e_j+m_ie_i)\cdot e_i=\sum_{\{j\neq i:e_j\cdot e_i\neq 0\}}m_je_j\cdot e_i-m_i.$$ 

The equalities hold by the second bullet of Lemma \ref{combination}. 

Furthermore, by the third bullet of Lemma \ref{combination}, 
$$\sum_{\{j\neq i:e_j\cdot e_i\neq 0\}}m_je_j\cdot e_i-m_i\ge 1$$ for any $s_1+1\le i\le s$. Now, for $s_1+1\le i \le s$, $$m_ie_i(e-e_i)=m_ie_i\cdot e-m_ie_i^2\ge m_i+m_i=2m_i.$$

Combining these two cases, we have  proved the Lemma. 
\end{proof}

%%%%%%%%%%%%%%%%%%%%%%%%%%%%%%%%%%%%%%%%%%%%%%%%%%%
\section{Rational curves}

When $g_J(e)=0$, we get more precise information. 

%\subsection{Proof of Theorem \ref{emb-comp}}
\subsection{Tree of smooth components}

\begin{cor}\label{smooth component}
Suppose $e$ is a $J-$nef class with $g_J(e)=0$.
If $\Theta=\{(C_i,m_i), 1\leq i\leq l\}\in \mathcal M_{red, e}$ is connected, then
each irreducible component is a smooth rational curve.\end{cor}

\begin{proof}
Recall we have proved Theorem \ref{genus bound} in the last section, which says \eqref{bound} holds. 
Since $g_J(e)=0$ and $g_J(e_i)\geq 0$, it follows from \eqref{bound} that we must have $g_J(e_i)=0$ for each $i$. 
By the adjunction inequality \eqref{adj inequality},  each $C_i$ is a smooth rational curve.
\end{proof}

Now we show that $\Theta$ is a tree configuration. 

\begin{definition}\label{deftree}
A connected weighted graph  with $l$ vertices is called a tree if the sum of the labels of the edges is 
 $l-1$, which is the minimal number ensuring the graph to be connected.

A tree  graph is called a simple tree graph if further, each vertex has multiplicity $1$. 
\end{definition}

\begin{lemma} \label{tree1}
Suppose $g_J(e)=0$ and  $\Theta=\{(C_i, 1), 1\leq i\leq l\}$ is connected curve configuration with total class $e$, then $g_J(e_i)=0$ and the underlying graph is a simple tree.
\end{lemma}
\begin{proof}
This follows from the argument in Lemma \ref{simple}. More precisely, since 
$g_J(e)=0$,  the estimate \eqref{simple estimate} has to be an equality. 
\end{proof}

\begin{lemma}\label{tree2}
Suppose we apply any of the three curve combination moves  to  a connected, nef curve configuration. 
If  each new vertex of the resulting curve configuration has   $adj=-2$, then so does each replaced vertex of the initial  curve configuration. 

\end{lemma}

\begin{proof}
We use the notation $D_i$ as in \ref{notation}.

For the first two moves, since  
$adj(D_i)\geq -2$, 
 it follows from \eqref{move12} that 
$-2=adj(D_1+D_2)$ if and only if 
$$adj(D_1)=adj(D_2)=-2, \quad D_1\cdot D_2=1.$$

For the third move, 
since $adj(D_i)\geq -2$, $g_J(E)\geq 0$ and $D_i\cdot E>0$, 
it follows from \eqref{move3adj} that 
$-2=adj(D_i+(D_i\cdot E)E)$ 
if and only if 
$$adj(D_i)=-2,  \quad g_J(E)=0, \quad D_i\cdot E=1.$$
\end{proof}

\begin{lemma}\label{tree3}

Suppose $G=\{(e_i, m_i), 1\leq i\leq l\}$ is a connected, nef curve configuration  with  each vertex having $adj=-2$. 
Let $G'$ be  the curve configuration obtained from  $G$ by  a curve expansion or a curve combination. 
If $G'$ is a 
tree, so is $G$. 
\end{lemma}
\begin{proof}
 We only need to verify the change of the sum of  labels is no smaller than the change of number of vertices.
Let $e$ be the total class of $G$.

For a curve expansion, a vertex weighted by $(D, m)$ becomes $m$ vertices weighted by $(D, 1)$. 
The number of vertices  increases by $m-1$, and since $G$ is connected, 
the sum of the labels increases by $(m-1)D\cdot (e-D)\ge m-1$.

For curve combination move (i), the number of vertices  decreases by $1$. 
Suppose the  two replaced vertices are   $C_1$ and $C_2$. The sum of labels  decreases
by 
$$(e_{C_1}\cdot  \sum_{i\ge 3}e_{C_i}+e_{C_2}\cdot  \sum_{i\ge 3}e_{C_i}+e_{C_1}\cdot e_{C_2})-((e_{C_1}+e_{C_2})\cdot  \sum_{i\ge 3}e_{C_i})= e_{C_1}\cdot e_{C_2}= 1.$$
The last step is due to the  $adj=-2$ assumption  and Lemma \ref{tree2}.

For move (ii), the number of vertices is unchanged. Let the two replaced vertices be $C_1$ weighted by $(D_1, n_1)$, and $C_2$ weighted by $(D_2, n_2)$. 
Due to the  $adj=-2$ assumption and Lemma \ref{tree2}, $D_1\cdot D_2=1$. 
Hence $D_1\cdot D_1=-1$ because $D_1\cdot D_2\ge -D_1\cdot D_1>0$. 
By nefness, $e\cdot D_1\ge 0$. So $C_1$ should connect to vertices other than $C_2$. And the sum of labels would increase by $$2(n_1-n_2)-((n_1-n_2)+1)\ge 0.$$

For move (iii)$_t$, the number of vertices would decrease by $1$. The number of labels would increase at least by  
\begin{equation}\label{iii12}
\sum_{1\le i<j\le t}(D_i\cdot D_j+1)- (t+\sum_{1\le i<j\le t}D_i\cdot D_j)\ge \frac{t(t-3)}{2}\ge -1.\end{equation}
\end{proof}

\begin{prop}\label{graph tree}
Suppose $e$ is a $J-$nef class with $g_J(e)=0$.
If  $G$ is connected  curve configuration with class $e$ and at least 2 vertices, then 
  $G$ is a tree graph.
\end{prop}

\begin{proof}
If $m_i=1$ for each $i$, the assertion follows from Lemma \ref{tree1}. 

Otherwise, we apply the curve expansion and combination moves to get a connected, nef curve configuration $G'$ with class $e$ and satisfying all the three bullets in Lemma \ref{combination}.
Notice that since $g_J(e)=0$, Lemma \ref{multi 1} implies that each vertex of $G'$ has  multiplicity  one. Therefore $G'$ is a tree.

Then by Lemma \ref{tree3}, $G$ is a tree  as well. 
\end{proof}

\begin{cor}\label{tree}
Suppose $e$ is a $J-$nef class with $g_J(e)=0$.
If $\Theta=\{(C_i,m_i), 1\leq i\leq l\}\in \mathcal M_{red, e}$ is connected, then
the underlying weighted  graph is a tree.
\end{cor}

\subsection{Dimension bound}
Suppose $e$ is a $J-$nef class with $g_J(e)=0$.
If $\Theta=\{(C_i,m_i), 1\leq i\leq n\}\in \mathcal M_{red, e}$ is connected,
by Corollary \ref{smooth component} and Proposition \ref{graph tree}, the underlying curve configuration of $\Theta$
is a tree with each vertex having genus $0$.

\subsubsection{$l_G$}
In light of this, we introduce the following definition.

\begin{definition} \label{lG}
The dimension of a tree graph $G$ with  vertices weighted by $\{(e_i, m_i)\}$ and having genus 0 is 
defined to be 
$$l_G=\sum_{i=1}^n l_{e_i}.$$
\end{definition}

Recall that $l_{e_i}=\max\{\iota_{e_i}, 0\}$, and 
$\iota(e_i)$ is the $J-$dimension defined by \eqref{l}.  Since $g_J(e_i)=0$, $\iota_{e_i}$ is equal to  $e_i\cdot e_i+1$ by Lemma \ref{irr=smooth}.

We stratify $\mathcal M_{red, e}$ according to the underlying curve configuration. 
By Lemma  \ref{pdmanifold},  $l_G$ is the complex dimension of the stratum corresponding to the curve configuration $G$. 

Let $L=l_e$. By Lemma  \ref{pdmanifold}, $L$ is the  complex dimension of $\mathcal M_{irr, e}$, as long as
$\mathcal M_{irr, e}$ is nonempty. 

\begin{example} \label{centered graph}  We illustrate Definition \ref{lG} with a  graph having $5$ vertices and dimension $L-1$. We will use notations as in Example \ref{referee}. For the centered graph $G$ with center vertex $(a-2b, 1)$ and four vertices $(b, 1)$, its dimension $l_G=4$. In this case, $e=a+2b$ and $L=l_e=5$. 
So $G$ is a codimension $1$ graph.
\end{example}

For a $J-$nef class, we have the following fact. 

\begin{lemma}\label{reducible-dim}
Suppose $e$ is a $J-$nef class with $g_J(e)=0$ and $L=l_e$.
If $G$ is a connected curve configuration with class $e$ and $n\geq 2$ vertices, and vertices weighted by 
$\{(e_i, m_i), 1\leq i\leq n\}$,  then 
\begin{equation}\label{red-dim'}
\sum_{i=1}^n m_i l_{e_i}\leq L-1.
\end{equation}
\end{lemma}

Lemma \ref{reducible-dim} is to be proved below. 
Before proving it, we first state an important  corollary.
Since $l_{e_i}\geq 0$ we have 

\begin{cor}\label{reducible-dim'}
Suppose $e$ is a $J-$nef class with $g_J(e)=0$ and $L=l_e$.
If $G$ is a connected curve configuration with class $e$ and at least 2 vertices, then the dimension $l_G$  of 
the stratum indexed by $G$ satisfies the following bound, 
\begin{equation}\label{red-dim}
l_G\leq L-1.
\end{equation} 

\end{cor}
This  is an analogue of Proposition 3.4 in \cite{T1},
but valid for an arbitrary almost complex structure.

\subsubsection{Lemma \ref{reducible-dim}}
\begin{proof} [Proof of Lemma \ref{reducible-dim}] Notice that  by the assumption,
 $e$ is  $J-$effective and $J-$nef, so  $e\cdot e\geq 0$ by Lemma \ref{irr=nef}. Hence 
\begin{equation}
L=l_e=\iota_e=e\cdot e+1\geq 1.
\end{equation}

Let us first assume that $n=1$. In this case, since $G$ has least 2 vertices,  $m_1\geq 2$.   
By the adjunction formula, this is impossible if $e_1\cdot e_1 =0$. 
But if $e_1\cdot e_1>0$, then  $e\cdot e\geq m_1^2$, and $l_{e_1}=1+e_1\cdot e_1$. Therefore 
$$m_1l_{e_1}=m_1+m_1e_1\cdot e_1 =m_1+ \frac{1}{m_1}  e\cdot e <1+ e\cdot e=L.$$

Now we assume that $n\geq 2$. Then
\begin{equation}\label{l-e}
L=\iota_e=\sum_{i=1}^n m_ie_i\cdot m_ie_i + \sum_{i=1}^n m_i e_i\cdot(e- m_i e_i) +1.
\end{equation}
Since $G$ is connected and $n\geq 2$,  
\begin{equation} \label {m-i} m_i e_i\cdot(e- m_i e_i)\geq m_i
\end{equation}  for each $i$.

I. Let us start with the simple case where each $e_i$ has $e_i\cdot e_i\geq 0$.
 Then $l_{e_i}=\iota_{e_i}=1+e_i\cdot e_i$ for each $i$,  and $m_i^2e_i\cdot e_i\geq  m_ie_i\cdot e_i$. 
 By \eqref{l-e}  
and \eqref{m-i},
 $$L\geq 1+\sum_{i=1}^n m_il_{e_i}.$$

II. General case.   Use $1, \cdots, k$ to label the vertices  whose class has self-intersection at most $-1$. 
Notice that  $l_{e_i}=0$ for $i=1, \cdots, k$.

Since $G$ is connected, $e_j\cdot (e-m_je_j)\geq 1$ for each $j\geq k+1$.
Therefore $L$ can be estimated as follows:
\begin{equation} \label{3terms}
\begin{array}{lll}
 L&=&1+e\cdot e\\
&&\\
 &=&1+\sum_{j=k+1}^n(m_j^2e_j\cdot e_j+ m_je_j\cdot (e-m_je_j))  + (\sum_{i=1}^k m_ie_i\cdot e)\\
&&\\
 &\geq&1+\sum_{j=k+1}^n m_jl_{e_j}+ (\sum_{i=1}^k m_ie_i\cdot e)\\
&& \\
 &=& 1+\sum_{j=1}^n m_jl_{e_j} + (\sum_{i=1}^k m_ie_i\cdot e).\\
 \end{array}
 \end{equation}
Finally, observe that, by the $J-$nefness of $e$,
the last   term $(\sum_{i=1}^k m_ie_i\cdot e)$  is non-negative.
\end{proof}

\subsection{Maximal dimension configurations} 
We assume $M$ has $b^+=1$. 
Let  $e$ continue to be a $J-$nef class with $g_J(e)=0$. 

If $G$ is a connected curve configuration with class $e$ and at least 2 vertices, we have  shown 
in the two previous subsections that  
$G$ is  a tree graph (Proposition \ref{graph tree}),  and $l_G$  is bounded above by 
$L-1$ (Corollary \ref{reducible-dim'}).

In this subsection we classify all possible maximal dimension
configurations with at least 2 vertices, i.e. configuration $G$ with 
$L= 1+l_G$. 

Let $G_-$ be the weighted subgraph containing each vertex whose class has self-intersection at most $-1$. Use $V_1, \cdots, V_k$ to label these vertices. 
Let $G_+$ be the weighted subgraph containing remaining vertices, use $V_j$ with $j\geq k+1$ to label these vertices. 

\begin{lemma}\label{G+}
If $l_G=L-1$ then  $m_j=1$ for $j\geq k+1$. Namely, the $G_+$ part is simple. 
\end{lemma}

\begin{proof}
Since $e$ is $J-$nef and  $l_{e_j}=0$ for $j\leq k$,    it follows from  (\ref{3terms}) that $L-1\geq \sum_{j=k+1}^n m_j l_{e_j}$.
Since $l_{e_j}\geq 1$ for each $j\geq k+1$ and $l_G=\sum_{j=k+1}^n l_{e_j}$, we have the desired claim.
\end{proof}

From the proof above we  actually have  two more consequences of \eqref{3terms} under the assumption that  $l_G=L-1$. 
The first one is 
%that     $m_i=1$ if  $k+1\leq i\leq n$, and 
\begin{equation} \label{positiveint}
e_i\cdot (e-e_i)=1, \quad \hbox{if $k+1\leq i\leq n$},
\end{equation}
and the second one is 
\begin{equation} \label{negativeint}
e\cdot e_i=0, \quad \hbox{if  $i\leq k$.}
\end{equation}

\subsubsection{When $G_-$ is empty}

\begin{lemma}\label{caseI}
If $l_G=L-1$ and  $G_-$  is empty, then $n=2$, and $m_1=m_2=1$.
\end{lemma}

\begin{proof}
In this case,  by Lemma \ref{G+}, $m_i=1$ for each $i$.  Moreover,   $e_i\cdot (e-e_i)=1$ for each $i$.  Since $G$ is connected, this is possible only if  $n=2$. 

\end{proof}

In the following we assume that $G_-$ is not empty. We first show that $G$ contains a centered subgraph. 
\subsubsection{A centered subgraph}

A simple tree graph is called centered if there is a vertex, called the center vertex,  which is adjoined to all  the other vertices.   Note that the graph in Example \ref{centered graph} is centered. 

\begin{lemma}\label{distinctivecurve}
 Assume  $G_{-}$ is non-empty and $l_G=L-1$. 
Then 
\begin{itemize}
\item The vertices in $G_+$  have the same weight with  
 $m_i=1$ and $e_i\cdot e_i=0$. 
 
\item  There is only one  vertex in 
$G_-$ which is adjoined to  the vertices in $G_{+}$.  Denote this vertex by $V_1$. $V_1$  has multiplicity one, and 
its class has self-intersection less than or equal to $k-n$.

\item The weighted subgraph consisting of the vertex $V_1$ and vertices in $G_+$ is  a centered graph with $V_1$ as the center. 

\item The weighted subgraph $G_-$ is connected. 
\end{itemize}
\end{lemma}

\begin{proof}

 For each $i$ with $k+1\leq i\leq n$, $l_{e_i}\geq 1$. 
  Since $l_G=L-1$, 
   % by  \eqref{3terms} , then 
 %   $m_i=1$ if  $k+1\leq i\leq n$, and 
%\begin{equation} \label{positiveint}
%e_i\cdot (e-e_i)=1, \quad \hbox{if $k+1\leq i\leq n$}.
%\end{equation}
it follows from Lemma \ref{G+} and \eqref{positiveint}  that, if  there are  $i\neq i'\geq k+1$ such that $e_i\cdot e_{i'}\neq 0$, 
then  $V_i$ and $V_{i'}$ are not adjoined to any other vertex.   But this is impossible
since $G_{-}$ is non-empty and $G$ is connected. 

Hence,  for $i\ge k+1$,  the vertices $V_i$ are not adjoined to each other. Since $b^+(M)=1$, by light cone lemma, $e_i\cdot e_i=0$ for each $i\geq k+1$,  and  for $i, i'\geq k+1$, $e_i= \alpha e_{i'}$.
By the adjunction formula, $\alpha=1$ for any pair.  We have established the first bullet.

By the first bullet, the vertices in $G_+$ are disjoint. It follows from Lemma \ref{G+} and  \eqref {positiveint} that,  each vertex in $G_+$ is adjoined to a unique vertex in $G_-$, and this vertex in $G_-$  has to have multiplicity one.
Since the classes of vertices in $G_+$ are the same, the vertices in $G_+$ are actually adjoined to the same vertex in $G_-$. Denote this vertex by $V_1$. 

%From \eqref{3terms},  we also have 
%\begin{equation} \label{negativeint}
%e\cdot e_i=0, \quad \hbox{if  $i\leq k$.}
%\end{equation}
It follows    from \eqref{negativeint}   that the class of $V_1$ has self-intersection less than or equal to  $k-n$. We have now established both the second and the third bullets. 

For the last bullet, it is a consequence of the second bullet since $G$ is connected. 
\end{proof}

Next we show that $G$ is a special kind of centered graph when $G_{-}$ is not empty and there are no $-1$ vertices.

\subsubsection{When $G_{-}$ is not empty and there are no $-1$ vertices}

\begin{lemma}\label{unique vertex} 
Suppose  $l_G=L-1$, $G_{-}$ is not empty and there are no $-1$ vertices.
Then $G_-$ contains a unique vertex $V$.  Furthermore, if $e_V\cdot e_V=  -b$,  then  $G$ is a centered graph with $b$ teeth. 

\end{lemma} 
\begin{proof} We first show that all vertices of $G$ have multiplicity $1$. 
By the first bullet of Lemma  \ref{distinctivecurve},  this is true for any vertex in $G_+$. 
Since  there are no $-1$ vertices, no curve combination move is 
needed to achieve the configuration described in Lemma \ref{combination}. Apply Lemma \ref{multi 1} to conclude that all
vertices of self-intersection less than $-1$ also have multiplicity 1.

Now we show that   every vertex in $G_-$ is adjoined to 
 at least two vertices of $G$.
Since every vertex of $G$ has multiplicity one, and  each edge of $G$ has label  $1$ by Proposition \ref{graph tree}, we see that from \eqref{negativeint}, 
once a vertex in $G_-$ is adjoined to only one other vertex of $G$, its self-intersection should be $-1$. But this is excluded by our  assumption.

By the second bullet of Lemma \ref{distinctivecurve} there is  
  only one vertex $V_1$ in $G_-$ which is adjoined to vertices in  $G_+$. So  any other vertex in $G_-$ is adjoined to  at least two vertices in $G_-$.   
By the last bullet of Lemma \ref{distinctivecurve},    $G_-$ is connected. 
   So if $G_-$ has more than one vertex, $V_1$ is adjoined to at least one
  another vertex in $G_-$. 

 Thus, if $k\geq 2$,  twice of the number of edges in $G_-$ is at least 
 $$2(k-1)+1=2k-1>2(k-1).$$ 
This means that  there must be a cycle in the weighted subgraph $G_-$. This implies that there is a cycle in $G$ as well, which contradicts Proposition  \ref{tree}. Hence, there is only one vertex  in $G_-$ . 

Finally, we conclude that $G$ is a centered graph by the third bullet of Lemma  \ref{distinctivecurve}.
\end{proof}

The remaining  case is that $G_-$ contains $-1$ vertices. We start with the following observation. 

\subsubsection{$\tilde G$ in Lemma \ref{combination}}

\begin{lemma}\label{after rearrangement}
Suppose $\tilde G$ satisfies all the three bullets  in Lemma \ref{combination} and $l_{\tilde G}=L-1$ or $L$. Then $\tilde G$ contains no $-1$ vertices. 
\end{lemma}
\begin{proof} 
If $l_{\tilde G}=L$, then $\tilde G$ has only one vertex by Corollary \ref{reducible-dim'}. This vertex is not a $-1$ vertex since
its class is just $e$, which is assumed to be $J-$nef. 

Now let us assume that $l_{\tilde G}=L-1$. 
Notice that, if there is a $-1$ vertex $E$ in $ \tilde G$,  then by the second and the third bullets 
of Lemma \ref{combination}, we have 
 $e_{\tilde G}\cdot E>0$. But this contradicts to  \eqref{negativeint}. 
\end{proof}

In light of Lemma \ref{after rearrangement},   we next   analyze how $l_G$ changes under  curve  moves.
\subsubsection{$l_G$ under curve moves}

\begin{lemma}\label{lG-expan}
Let $G'$ be obtained from $G$ by a curve expansion. Then $l_G<l_{G'}$.

\end{lemma}
\begin{proof}
This is clear  since $$D\cdot D+1<n(D\cdot D+1)$$ if $D\cdot D\ge 0$ and $n>1$. 
\end{proof}

\begin{lemma}\label{dimnd} Let $G'$ be obtained from $G$ by a  curve combination, which is not of type (i)$_1$ with $D_1\cdot D_1\ne -1$. 
Then $l_G \leq l_{G'}$. Furthermore, $l_{G}=l_{G'}$ 
if and only if the class of each new  vertex of $G'$ has negative self-intersection. 
In particular, under such a move,  $l_G=l_{G'}$   if $l_G=L-1$ and $G'$ has more than one vertices.
\end{lemma}
\begin{proof}

For move (i)$_n$ with $n\ge 2$, the part modified has $l_{D_1}+l_{D_2}=0$.

For move (iii), the part modified has $\sum_{i=1}^tl_{D_i}+l_{E}=0$.

In these two cases, $l_G \leq l_{G'}$ since a new vertex $V$ always has $l_V\ge 0$. The equality $l_G = l_{G'}$ holds if and only if $V$ has negative self-intersection.

For move (ii), since $n_1\ge 2$, we have $l_{D_1}=0$. Meanwhile, $$(D_1+D_2)\cdot (D_1+D_2)>D_2\cdot D_2,$$which implies $l_G \leq l_{G'}$. The equality $l_G=l_{G'}$ holds if and only if $(D_1+D_2)\cdot (D_1+D_2)<0$.

For move (i)$_1$ with  $D_1\cdot D_1=-1$, we have $$D_2\cdot D_2+1<(D_1+D_2)\cdot (D_1+D_2)+1.$$ Similarly, $l_G=l_{G'}$ if and only if $(D_1+D_2)\cdot (D_1+D_2)<0$.

The last statement follows from Corollary \ref{reducible-dim'}.
\end{proof}

\subsubsection{$\tilde G$ in Lemma \ref{combination} revisited}
 Given the two lemmas above, we have the following  more precise description of $\tilde G$.

\begin{lemma}\label{tilde G}
Suppose $G$ contains a $-1$ vertex and $l_G=L-1$. We apply curve  moves   as in  Lemma \ref{combination}
 to adjust $G$ to a configuration $\tilde G$ satisfying all the three bullets there. 
 Let $G_p, ..., G_1$ be the intermediate graphs. 
 Then 
 
 \begin{itemize} 
 
\item  $l_{G_i}=L-1$, 

\item   $l_{\tilde G}=L-1$ or $L$,

\item  There are no $-1$ vertices in $\tilde G$,

\item If $\tilde G$ has at least 2 vertices,  then it is either a  graph with precisely 2 vertices as in   Lemmas \ref{caseI},  or a centered graph as in  \ref{unique vertex},

\item      $\tilde G$ is a simple tree graph.

\item $\tilde G_-$ contains at most one vertex.

\end{itemize}
  \end{lemma}
\begin{proof}
Notice that   the curve combinations in Lemma \ref{combination}  only involve $-1$ vertices, the first and second bullets  follow from  Lemmas \ref{lG-expan},  \ref{dimnd} and Corollary \ref{reducible-dim'}.

The third bullet follows from the second bullet and Lemma \ref{after rearrangement}. 

We now prove the  fourth bullet. If $\tilde G$ has at least 2 vertices, by Corollary \ref{reducible-dim'} and the second bullet, $l_{\tilde G}=L-1$. 
Since $\tilde G$ contains no $-1$ vertices by the third bullet, the statement follows from 
 Lemmas \ref{caseI},  \ref{unique vertex}.
 
The last two bullets follows from the fourth bullet.

\end{proof}

We will see that only the following restricted moves, which we call  combinatorial blow-downs, are needed to obtain $\tilde G$ from $G$.  
 
\subsubsection{Combinatorial blow-downs}

\begin{definition}\label{co blowdown}
A simple combinatorial blow-down applied to a weighted graph $G$ is the following  process of removing a $-1$ vertex $V$ of genus $0$.
\begin{enumerate}
\item Either  $V$ is weighted by $(v, t)$ and adjoined to only one vertex $U$ weighted by $(u, t)$ with $u\cdot v=1$, then in the new graph these two vertices are removed and a new vertex $U'$ weighted by $(u+v, t)$ is added.
\item Or  $V$ is weighted by $(v, t_1+t_2)$ and adjoined to exactly two vertices $U_1$ weighted by $(u_1, t_1)$ and $U_2$ weighted by $(u_2, t_2)$ with edges labeled by one, i.e. $v\cdot u_1=v\cdot u_2=1$, then these three vertices are replaced by two new vertices $U_1'$ weighted by $(u_1+v, t_1)$ and $U_2'$ weighted by $(u_2+v, t_2)$. 
\end{enumerate}
The inverse process is called a simple combinatorial blow-up.
\end{definition}
Geometrically, the first bullet corresponds to blowing up at a smooth point in the subvariety, 
the second bullet corresponds to blowing up at a transversal intersection point of two irreducible components.

\subsubsection{Each move is a combinatorial blow down}

\begin{lemma}\label{minimal graph}
Suppose $G$ contains a $-1$ vertex and $l_G=L-1$. Then after applying simple combinatorial blow-downs,  $G$ can be turned into a curve configuration $\tilde G$ with no $-1$ vertices.
There are two cases: 

\begin{itemize}
\item $\tilde G$ consists of  only  one vertex,   whose class has non-negative self-intersection. In this case, 
except for the last blow-down, all the  vertices involved in  blow-downs have classes
with  negative self-intersection.

\item $\tilde G$ is a centered graph. In this case, all the vertices involved in blow-downs have classes
with  negative self-intersection.
\end{itemize}
\end{lemma}

\begin{proof}
We apply curve  moves  to adjust $G$ to a configuration $\tilde G$ as in Lemma \ref{tilde G}. 

 Let $G_p, ..., G_1$ be the intermediate graphs. It is convenient to sometimes write $G=G_{p+1}$ and $\tilde G=G_0$. 

By Lemma \ref{tree2} and Lemma \ref{tree3},  for each $i\geq 0$, each $G_i$ is a tree graph of genus $0$ vertices.  Further, by Lemma \ref{tilde G}, $G_0=\tilde G$ is a {\it simple} tree graph and $\tilde G_-$ contains at most one vertex. 

 For $i\geq 1$, $G_i$ contains at least 2 vertices, one of them is a $-1$ vertex. In fact, the move from $G_i$ to $G_{i-1}$ involves a
$-1$ vertex of $G_i$.  By Lemma  \ref{tilde G},  $\tilde G$ has no $-1$ vertices, 
\begin{equation} \label{lGi} l_{G_i}=L-1  \hbox{    for $i>0$, \quad and    \quad }    l_{\tilde G}=L-1  \hbox{ or $L$.} 
\end{equation}

 \noindent{\bf No expansion moves:}  
First we notice that curve expansion never occurs in the process. 
By Lemma \ref{lG-expan} and \eqref{lGi}, it could only possible occur in the last step,  from $G_1$ to $\tilde G$ and  when $l_{\tilde G}=L$. If this is the case,
then $\tilde G$ has more than one vertex since expansion increases the number of vertices. However, this is impossible since    in this case  $\tilde G$ consists of a single vertex with multiplicity one due to the assumption $l_{\tilde G}=L$.

\noindent {\bf The  move from $G_q$ to $G_{q-1}$ for $q\geq 2$:}
We know it is a {\it combination} move involving a $-1$ vertex. We   will show that it is a  simple combinatorial blow down.

We first make a general observation. 
Notice that for $q\geq 2$,  $l_{G_q}=l_{G_{q-1}}$.  Therefore, by Lemma \ref{dimnd},  the classes of  the new vertices in $G_{q-1}$ have  negative self-intersection.  

I.   Suppose for some $q\geq 2$ the move from $G_q$ to $G_{q-1}$ is a type (i) move. 
Then there are two adjoined vertices of the tree graph $G_1$, 
 $U_1$ weighted by $(u_1, t)$ and  $U_2$ weighted 
by $(u_2, t)$,
one of them, say $U_2$,  is a $-1$ vertex, and they are replaced by  a vertex $U$ of $\tilde G$ weighted by 
$(u_1+u_2, t)$.     Clearly,  this move is just  a  type (1)  simple combinatorial blow-down.

Notice that $(u_1+u_2)\cdot (u_1+u_2) =u_1\cdot u_1+2-1=u_1\cdot u_1+1$ is negative. Therefore $u_1\cdot u_1$ is negative as well.

II. Moves (ii) are not needed. 
Suppose for some $q\geq 2$ the move from $G_q$ to $G_{q-1}$ is a type (ii) move in the proof of Lemma \ref{combination}. 
Since such a move is applied to a $-1$ vertex and another vertex whose class has self-intersection at least $-1$, 
 the class of one new vertex of $G_{q-1}$ has non-negative self-intersection. But this is impossible.

III.  Suppose for some $q\geq 2$ the move from $G_q$ to $G_{q-1}$ is a type (iii)$_t$ move. 
Then there are $t$  
vertices $U_i$ of $G_1$ weighted by $(u_i, n_i)$, $1\le i\le t$, and a $-1$ vertex $V$ of $G_1$ weighted by $(v, n_0)$, 
such that $$u_i\cdot u_i\leq -2, \quad \hbox{and}  \quad \sum_j n_i (u_i\cdot v)=n_0.$$
They are replaced by  vertices $W_i$ weighted by $(u_i+v, n_i)$.

Both $G_q$ and $G_{q-1}$ are tree graphs, and since the number of vertices of $G_{q-1}$ is 1 less than that of vertices of $G_q$, the number of edges of 
$G_{q-1}$ is also 1 less than that of labels of $G_q$. 
Apply the inequality \eqref{iii12},  we find that it is only possible that $t=1$ or $2$.

When $t=1$, the  move is also a type (i) move, so it is a type (1) blow-down.  As already shown, the classes of the involved vertices all have negative self-intersection.

Now assume that  $t=2$.   Notice that  $u_i\cdot v=1$ since $G_q$ is a tree graph. Hence this move  is a type (2)  simple combinatorial blow-down.
We just need to verify the classes of the involved vertices all have negative self-intersection. This is true for $U_1, U_2$ and $V$ by assumption. For $W_i$, this is also true since
$(u_i+v)\cdot (u_i+v)\leq -2+2-1=-1$. 

\noindent {\bf The move from $G_1$ to $\tilde G$:} 
The next step is to analyze  the  curve  move from $G_1$ to $G_0=\tilde G$.

I. Suppose this step is a type (i) move. 
We have already shown that it is  a type  (1) combinatorial blow-down.
Here we have 

Since $\tilde G$ is simple, $t$ can only be equal to $1$.

If $(u_1+u_2)\cdot (u_1+u_2)\geq 0$,
then $l_{u_1+u_2}\ge l_{u_1}+l_{u_2}+1$ and hence $l_{\tilde G}\geq l_{G_1}+1=L$.  By Corollary \ref{reducible-dim'}, $\tilde G$  consists of  a single vertex weighted by $(u_1+u_2, 1)$. 
Thus this case  corresponds to the  first bullet of Lemma \ref{minimal graph}.

If $(u_1+u_2)\cdot (u_1+u_2)<0$,  then $l_{\tilde G}=l_{G_1}=L-1$ and $\tilde G_-$ is not empty.  
Hence $\tilde G$ is a centered graph as in  Lemma \ref{unique vertex}.
Moreover, notice also that $u_1\cdot u_1<- u_2\cdot u_2-2u_1\cdot u_2=1-2=-1$. 
Thus this  move 
is  a combinatorial blow-down and the classes of all the vertices involved have negative self-intersection.  This case corresponds to   the second bullet of Lemma \ref{minimal graph}.

II.  Suppose this step  is a type (ii) move. 
Then there are two adjoined vertices of the tree graph $G_1$, 
 $U_1$ weighted by $(u_1, t_1)$ and  $U_2$ weighted 
by $(u_2, t_2)$ with $t_1>t_2$ and  
\begin{equation} \label{u1} u_1\cdot u_1\geq -u_1\cdot u_2=-1.
\end{equation}
One of them  is a $-1$ vertex, and  
they are replaced by  a vertex $U$ weighted by 
$(u_1+u_2, t_2)$ and a vertex $V$ weighted by $(u_1, t_1-t_2)$.

If  $U_1$ is  a $-1$ vertex.
then the vertex $V$ of $\tilde G$ is  a $-1$ vertex. But $\tilde G$ doesn't contain any $-1$ vertex, so $U_1$ is not a $-1$ vertex, and from \eqref{u1} we must have  $u_1\cdot u_1\geq 0$. 
We then conclude that  $t_1=1$ by Lemma \ref{G+}. But then $t_2=0$ since $t_2<t_1$. This simply means that this step cannot be a type (ii) move.

III. Suppose this step is a type  (iii)$_t$ move. 

  Since $G_1$ is a tree graph,  $u_i\cdot v=1$. 
Hence $$(u_i+v)\cdot (u_i+v)\leq -2 +2-1=-1$$ for $1\leq i\leq t$. Since $\tilde G_-$ contains at most one vertex, we have $t=1$. 
Thus in this case $\tilde G$ is a centered graph. 

Since $\tilde G$ is a simple graph, 
we have $n_1=1=n_0$.  In other words, this move  is actually a (special case of) type (i)$_1$ move, and this case corresponds to the first bullet of Lemma \ref{minimal graph}.

We thus complete our proof.
\end{proof}

\subsubsection{Summary}
Thanks to  Lemmas \ref{caseI},  \ref{unique vertex},   \ref{minimal graph}, 
we can completely describe those $G$ with $l_G=L-1$.

\begin{prop}\label{final}
Suppose $b^+(M)=1$,  $e$ is a $J-$nef class with $g_J(e)=0$. 
Let $G$ be a connected curve configuration with class $e$ and $l_G=L-1$.  

If $G$ does not contain any $-1$ vertex,  then $G$ is a simple graph tree. Moreover, it is either a  graph with precisely 2 vertices as in   Lemmas \ref{caseI},  or a centered graph as in  \ref{unique vertex}.
If $G$ contains a $-1$ vertex, then $G$ is as described in Lemma \ref{minimal graph}.
\end{prop}
%%%%%%%%%%%%%%%%%%%%%%%%%%%%%%%%%%%%%%%%%%%%%%%%%%%%%%%%%%%%
\subsubsection{Maximal dimension strata of $\mathcal M_{red, e}$}
We translate Proposition \ref{final} to the description of 
the maximal dimension strata of  $\mathcal M_{red, e}$. 

To state the result,  for $\Theta \in \mathcal M_{red, e}$, write $\Theta=\Theta_{+}\cup \Theta_{-}$, 
where $\Theta_{-}$ contains each  pair  $(C,m)$ with  $e_C\cdot e_C\leq -1$. Label  the pairs in $\Theta_{-}$ by $(C_1, m_1),\cdots, (C_k, m_k)$.
A pair $(C, 1)$ is called a $-1$ component if $C$ is a smooth rational curve with $e_C\cdot e_C=-1$. 

\begin{theorem}\label{cd1conf}
Suppose $b^+(M)=1$,  $e$ is a $J-$nef class with $g_J(e)=0$. 
Let $\Theta=\{(C_i,m_i), 1\leq i\leq n\}$ be a subvariety in $\mathcal M_{red, e}$, and $L=l_e$. 
Then $L= 1+\sum_{i=1}^n l_{e_{C_i}}$  only if  each $m_i$ is equal to  $1$.

Moreover, $\Theta$ satisfies the following conditions: 

$\bullet$ If $\Theta_-$ is empty then 
$n=2$,   $e_{C_1}\cdot e_{C_2}=1$,  $e_{C_i}\cdot e_{C_i}\geq 0$.

$\bullet$ If $\Theta_{-}$ is not empty and there is no $-1$ component, then 
$\Theta_{-}$ consists of a unique component $C_1$ with $e_{C_1}\cdot e_{C_1}=1-n\leq -2$, and 
$\Theta_{+}$  consists of  at least $n-1\geq 2$ components $C_i, i\geq 2$, with $e_{C_i}=\cdots= e_{C_n}$ and $e_{C_i}\cdot e_{C_i}=0$.
Moreover,  $e_{C_1}\cdot e_{C_2}=1$. In short,   
$\Theta$ is a comb like configuration.  

$\bullet$ Suppose $\Theta_{-}$ contains a $-1$ component.  Then there are two cases.

\begin{enumerate}
\item   $\Theta$ is a successive  blow-up of
a  smooth rational curve with non-negative self-intersection. And from the second blow-up, 
we only blow up at a point not lying in any component  with non-negative self-intersection (there is at most one such component).

\item  $\Theta$ is a successive   blow-up
of a comb like configuration in the second bullet at points  in $C_1$ and its proper transforms.  
Moreover, the blow-up points do not lie in $C_i$ for each $i\geq 2$. 
\end{enumerate}
Conversely, if $\Theta$ is as in any bullet above, then $L= 1+\sum_{i=1}^n l_{e_{C_i}}$.
\end{theorem}

The next example  illustrates the two cases containing a $-1$ component.  Notations are as in Example \ref{referee}.

\begin{example} %What  (1) of the third bullet  means is that, starting from the second blow-up, we only blow up at a point in a component with negative self-intersection. 
For case (i) suppose we start off with a smooth rational  curve in class $b$.  Since $b\cdot b=0$ we have  $L=1$. 
Then the proper transform of the $b$ curve has negative self-intersection after one blow-up, so we can do any further blow-ups. However, if we start off with a smooth rational curve $C$  in class $a+b$, then the blow-ups are restricted: only the first one can be at a point meeting the curve $C$ (or its proper transform). 

%What (2) of the third bullet  means is that, all the blow-ups, from the second one on, occur at some  point not lying in the proper transform of the original configuration. Equivalently,  it means that  we successively blow up in the union of  components of negative self-intersection.

For case (2) suppose we start off with the comb like configuration in Example \ref{centered graph}
with $C_1$ being the smooth curve in class $a-2b$. Then we can blow up  at two different points in $C_1$ not lying in any of the four  curves in class $b$. 
\end{example}

\begin{remark} 
By Proposition \ref{existence}, the condition $b^+(M)=1$ is automatic if $J$ is assumed to be tamed.
\end{remark}

%%%%%%%%%%%%%%%%%%%%%%%%%%%%%%%%%%%%%%%%%%%%%%%%%%%%%%%%%%%%%%%%%%%%%%%
\subsection{Tamed $J$}

In this subsection $J$ is assumed to be a tamed almost complex structure on $M$. 

Let $e$ be a class in $S_{K_J}$.  
 Recall that $S_{K_J}$ is 
the set of  $K_J-$spherical classes, defined to be  $\{e\in S|g_J(e)=0\}. $
Here $S$ is the set of homology classes which are represented by smoothly embedded spheres.

\subsubsection{Connectedness and $J-$nef class}

\begin{prop}\label{connected}
Suppose $e\in S_{K_J}$ and $\Theta=\{(C_i, m_i)\}\in \mathcal M_e$. If $e\cdot e_{C_i}\ge 0$ for each $i$, then $\Theta$ is connected and each component $C_i$ is a smooth rational curve.
\end{prop}
\begin{proof}

First observe that $e\cdot e=e\cdot \sum_i e_{C_i}\geq 0$. Hence $b^+(M)=1$ by Proposition \ref{existence}.

Suppose $\Theta$ is not connected. Since $b^+(M)=1$, and each class $e_{C_i}$ is nontrivial by Lemma \ref{non-trivial}, 
 either $\Theta$ has a connected component $D$ with 
negative self-intersection, or it consists of $p\geq 2$ homologous connected components, $D_i$, with self-intersection $0$. 

The first case is impossible since $e\cdot e_{D}=e_D\cdot e_D<0$. 

In the second case, denote $e'=e_{D_i}$. Then  $-2=K_J\cdot e=K_J\cdot p e'$. 
Thus $p=2$. But $K_J\cdot e'=1$ and $e'\cdot e'=0$, which is impossible since $K_J$ is characteristic.

Since $\Theta$ is a nef configuration, Proposition \ref{emb-comp-graph} implies each component $C_i$ is a smooth rational curve.
\end{proof}

\begin{example}\label{not connected}
In $\mathbb C \mathbb P^2\#2\overline{\mathbb C \mathbb P^2}$, if $E_1-E_2$ is $J-$effective, then the  class
$3H-2E_2$ in $S_{K_J}$ is not $J-$nef, and there is a disconnected 
curve in this class with connected components in 
$3H-E_1-E_2$ and $E_1-E_2$. 
\end{example}

\begin{proof} [Proof of Theorem \ref{emb-comp}] 
 The first bullet follows from Corollary \ref{smooth component} and  
 Corollary \ref{tree}.
 
 The second bullet follows from Proposition \ref{connected}.
\end{proof}

\begin{proof} [Proof of Corollary \ref{embsphere}]

When $e\cdot e<0$, the conclusion follows from Lemma \ref{uniqueness}.

Suppose now  that $e\cdot e\geq 0$. Observe first that $g_J(e)=0$. 
Observe also that  by Lemma \ref{irr=nef}, $e$ is $J-$nef. Hence
the conclusion follows from Proposition \ref{connected}.

\end{proof}

\subsubsection{Remarks on Theorem \ref{emb-comp} and Corollary \ref{embsphere}} \label{rk-smooth}
Examples \ref{-2K} and \ref{not connected} demonstrate that  $J-$nefness is necessary for  Theorem \ref{emb-comp}.

As mentioned in the introduction,  Theorem \ref{emb-comp} and Proposition  \ref{reducible-dim}  are  crucial in \cite{LZ12} 
in  applying  Taubes's subvarieties-current-form's approach to Donaldson's  tamed versus compatible question for an arbitrary tamed almost complex structure on rational manifolds.

Various versions  of Corollary  \ref{embsphere} have appeared in the  literature. 
When $J$ is integrable,  it is used in the classification of rational surfaces in \cite{GH}. On  page 521 in \cite{GH}, a simple  argument is given, but unfortunately, it is not correct. \footnote {Since the third term there should be $\frac{1}{2}\sum_{\nu\neq \nu'}a_{\nu'}C_{\nu}\cdot C_{\nu'}-\sum a_{\nu}+1$, which is not necessarily nonnegative.} Presumably there is a substitute for this argument, but 
we have not  been able to find out if our result is new in this case.

When $M=\mathbb C \mathbb P^2\#k\overline{\mathbb C \mathbb P^2}$ with $k\le 9$, 
it is shown  in \cite{Z} that for any tamed $J$, 
 an irreducible curve $C$ with $C\cdot C<0$ must 
be a smooth rational curve. One can easily deduce Corollary \ref{embsphere} for such manifolds from this fact.

For a generic tamed $J$, McDuff \cite{McD} provided a more intricate argument and established  several special cases,  which are essential  in characterizing  rational symplectic $4-$manifolds
in  terms of embedded symplectic spheres with positive self-intersection. 
Recently, McDuff and Opshtein  in \cite{McD1}  investigate the structure of generic  pseudo-holomorphic curves in a relative setting. The reducible $J-$holomorphic curves considered there
are limits of irreducible embedded $J'-$holomorphic curves for generic $J'$ converging to $J$,
and hence the Gromov compactness can be applied to bound the topological type of the reducible
subvariety. 
%As a by product, they  obtains the same statement as in Corollary 
%\ref{embsphere} in this setting (see Corollary 1.7 in \cite{McD1}).
A related general remark is   that Corollary \ref{embsphere} applies  to an  arbitrary 
%$J-$holomorphic 
subvarieties in the moduli space.
If the subvariety lies in a connected component of the moduli space which contains a smooth rational curve,  one might be able to  prove the result using the Gromov compactness. 
%the stable curve approach in Gromov-Witten theory. 

%%%%%%%%%%%%%%%%%%%%%%%%%%%%%%%%%%%%%%%%%%%%%%%%%%%

\end{document}